\documentclass[a4paper,11pt,reqno]{amsart}
\RequirePackage[table]	{xcolor}			
\definecolor{green1}{RGB}{0,107,28}
\definecolor{green2}{RGB}{17,85,35}
\definecolor{green3}{RGB}{0,77,20}
\definecolor{green4}{RGB}{33,166,68}
\definecolor{green5}{RGB}{60,166,88}

\definecolor{blue1}{RGB}{3,56,91}
\definecolor{blue2}{RGB}{16,51,73}
\definecolor{blue3}{RGB}{1,40,66}
\definecolor{blue4}{RGB}{35,109,157}
\definecolor{blue5}{RGB}{59,118,157}
\RequirePackage[
	final,							
	colorlinks=true,
	urlcolor=blue3,						
	linkcolor=blue3,					
	citecolor=green4,					
	hypertexnames=false,					
	hyperfootnotes=false,					
	]{hyperref}						

\usepackage{a4wide}
\usepackage{latexsym}
\usepackage{amsopn,amscd}
\usepackage{amsfonts}
\usepackage{amssymb}
\usepackage{amsthm}
\usepackage{amsmath}
\usepackage[all]{xy}
\usepackage{scalefnt}

\usepackage{hyperref}
\usepackage{enumitem}
\setitemize{leftmargin=*}
\setenumerate{leftmargin=0.45cm}

\newtheorem{thm}{Theorem}[section]
\newtheorem{cor}[thm]{Corollary}
\newtheorem{lem}[thm]{Lemma}
\newtheorem{prop}[thm]{Proposition}

\newtheorem{examp}[thm]{Example}
\newtheorem{rema}[thm]{Remark}

\DeclareFontFamily{OT1}{pzc}{}
\DeclareFontShape{OT1}{pzc}{m}{it}{<-> s * [1.200] pzcmi7t}{}
\DeclareMathAlphabet{\mathpzc}{OT1}{pzc}{m}{it}

\def\mathcat{\mathpzc}

\def\xx{x}
\def\xxx{\partial}
\def\hh{s}
\def\hhh{h}
\def\Ho{\mathrm H}

\def\gg{\mathfrak g}

\def\HH{\mathrm H}
\def\TT{\mathrm T}

\def\CC{\mathbb C}
\def\AAA{\mathcal A}

\def\DDD{\mathcat D}

\def\MMM{\mathcal M}

\def\HHH{\mathcal H}

\def\NNN{N}
\def\MMM{M}
\def\dell{\mathcal D}
\def\ppi{\pi}

\def\AP{\mathcal P}
\def\AH{\mathcal H}
\def\Nsddata#1#2#3#4#5{
   (#4
     \begin{CD}
      \null @>#2>> \null\\[-3.2ex]
      \null @<<#3< \null
     \end{CD}
    #1, #5)
}
\long
\def\MSC#1\EndMSC{\def\arg{#1}\ifx\arg\empty\relax\else
      {\par\narrower\noindent
      2010 Mathematics Subject Classification. #1\par}\fi}

\long
\def\KEY#1\EndKEY{\def\arg{#1}\ifx\arg\empty\relax\else
    {\par\narrower\noindent
      Keywords and Phrases: #1\par}\fi}

\title 
{Pseudo Maurer-Cartan perturbation algebra and pseudo perturbation lemma}

\author{Johannes Huebschmann  }
\address{
\newline
Universit\'e de Lille - Sciences et Technologies 
\\
D\'epartement de Math\'ematiques\\
\newline CNRS-UMR 8524,
Labex CEMPI (ANR-11-LABX-0007-01)
\\
\newline
59655 Villeneuve d'Ascq Cedex, France\\
\newline
Johannes.Huebschmann@univ-lille.fr
 }

\date{\today}
\numberwithin{equation}{section}

\begin{document}
\setcounter{page}{1}

\maketitle
\medskip
\centerline
{To Nodar Berikashvili}
\medskip

\begin{abstract} We introduce the
pseudo Maurer-Cartan perturbation algebra, establish
a structural result 
and explore the structure of this algebra.
That structural result
entails, as a consequence, what we refer to as
the pseudo perturbation lemma.
This lemma, in turn, implies the
ordinary perturbation lemma.
\end{abstract}

\MSC 

\noindent
Primary: 
16E45 %Differential graded algebras and applications

\noindent
Secondary: 
17B55 %Homological methods in Lie (super) algebras
18G35 %Chain complexes
18G50 %Nonabelian homoloical algebra
18G55 %Homotopical algebra
55R20 %Spectral sequences and homology of fiber spaces
55U15 %Chain complexes
\EndMSC

\KEY Berikashvili's functor $\DDD$, homological perturbation theory, 
deformation theory, abstract gauge theory, pseudo Maurer-Cartan
perturbation algebra, pseudo perturbation lemma

 \EndKEY
{\tableofcontents}

\section{Introduction} 

It is a pleasure to dedicate this paper to Nodar Berikashvili.
In \cite{MR1710565} I pointed out that there is an intimate
relationship between Berikashvili's functor $\DDD$ and
deformation theory. In particular,
cf. \cite[Section 5]{MR1710565},
there is a striking similarity between
Berikashvili's functor $\DDD$
and a 
functor written 
in the deformation theory literature
as $\mathrm{Def}_\gg$ for a differential graded Lie
algebra $\gg$.
Here I develop a small aspect of that relationship.
I introduce and explore the
{\em pseudo Maurer-Cartan perturbation algebra\/}. This algebra relates to
deformation theory in an obvious manner, and it so does
as well with regard  to 
Berikashvili's functor $\DDD$:
One can view the members of
the
pseudo Maurer-Cartan perturbation algebra
as operators on objects of the kind that lead to
Berikashvili's functor $\DDD$.

A recent result of Chuang and Lazarev \cite{chuanglazarev} 
shows that the ordinary perturbation lemma
is a consequence of a 
structural result for a  
differential graded bialgebra
that arises by abstracting from the
operators acting on what these authors refer to
as an {\em abstract Hodge decomposition\/}; see Section \ref{hpt}
below for the latter notion.
The underlying differential graded algebra results from extending
an observation
in \cite{MR1057939, MR1802006}.
I show here that a variant of the algebra 
in \cite{chuanglazarev},
the  pseudo Maurer-Cartan perturbation algebra,
leads to
the same kind of conclusion. 
Indeed,
a  
similar
structural result,
Theorem \ref{structural} below,
entails as well, as a consequence, the ordinary perturbation lemma.

The notion of abstract Hodge decomposition
is equivalent to that of contraction, 
a basic concept in homological perturbation theory.
A more general notion is this: A {\em pseudocontraction\/}
consists of a chain complex $\NNN$, 
a chain endomorphism $\tau \colon \NNN \to \NNN$,
and a homogeneous degree $1$ operator
$\hhh \colon \NNN \to \NNN$ such that $d \hhh + hd = \tau$ and
$\hhh^2 = 0$. 
Here
$\tau$ is not necessarily an idempotent
endomorphism
nor are the data subject to any annihilation property  (side condition)
beyond the vanishing of $\hhh^2$.
Abstracting from
the formal properties of the algebra of operators
acting on a pseudocontraction together with a perturbation
of the differential leads to
the
pseudo Maurer-Cartan perturbation algebra.
The pseudo Maurer-Cartan perturbation algebra
surjects non-trivially to
the corresponding algebra in \cite{chuanglazarev}
and hence recovers all the members of this algebra.
Thus the
pseudo Maurer-Cartan perturbation algebra
 yields all the relevant operators
that act on any chain complex arising from 
an abstract Hodge decomposition with a perturbation of the differential
or, equivalently, from
a contraction with a perturbation of the differential.
Theorem \ref{structural} below says that
a 
structural result 
which Chuang and Lazarev show to be valid
for the algebra they consider
still  holds formally
for the pseudo Maurer-Cartan perturbation algebra.
The structure of the
pseudo Maurer-Cartan perturbation algebra
is somewhat simpler than that of 
the corresponding algebra in \cite{chuanglazarev}:
There is no annihilation contraint
beyond the vanishing of the square of $\hhh$,
and $\tau$ is not necessarily an idempotent, which
is equivalent to the axiom $\pi \nabla = \mathrm{Id}$
imposed on a contraction
$  (\MMM
     \begin{CD}
      \null @>{\nabla}>> \null\\[-3.2ex]
      \null @<<{\pi}< \null
     \end{CD}
    \NNN, h) 
$;
see Section \ref{hpt} below.
The present terminology 
\lq\lq pseudo Maurer-Cartan perturbation algebra\rq\rq\ 
avoids confusion with the notions of {\em Maurer-Cartan algebra\/}
\cite{MR1425752} and of 
{\em multi derivation {M}aurer--{C}artan algebra\/}
\cite{MR3584886}.
A consequence of Theorem \ref{structural}
is the 
{\em pseudo perturbation lemma\/}.
Corollary \ref{pseudolem} 
and Corollary \ref{pseudolem2} below
spell out two versions thereof.
The pseudo perturbation lemma implies
the ordinary perturbation lemma, see
Section \ref{hpt} below.
The results of this paper 
admit extensions, not made precise here,
 relative to additional algebraic structure
like algebra or coalgebra structures, similar to such generalizations
in \cite{MR1109665}.

In \cite{MR3881491} I explained another small aspect  of the relationship
between 
Berikashvili's functor $\DDD$
and the functor $\mathrm{Def}_\gg$ for a differential graded Lie
algebra $\gg$.
Also, working out the connections with
\cite{schlstas, MR517083} would be an exceedingly attractive project.

\section{Preliminaries}
The ground ring $R$ is a
commutative ring with unit.
Henceforth \lq\lq chain complex\rq\rq, \lq\lq algebra\rq\rq\ etc.
means 
$R$-chain complex, $R$-algebra, etc.
As in classical differential homological algebra, cf., e.g., 
\cite{MR0365571},
we denote the identity morphism on an object by the same
symbol as the object.

\section{Pseudo perturbation algebra}

Let  $\AH$ be the differential graded algebra
generated by $\hh$ and $\tau$  
of degrees $1$ and zero, respectively,
with differential (lowering degree by $-1$) written as $D$,
subject to
\begin{align}
D\hh&= \tau,
\label{pc1}
\\
\hh^2 &= 0.
\label{pc2}
\end{align}
We refer to $\AH$ as the 
{\em pseudocontraction algebra\/}.

\begin{prop}
The algebra generators $\tau$ and $\hh$ of $\AH$ commute. Hence
the graded algebra that underlies $\AH$ decomposes as $\Lambda[\hh] \otimes R[\tau]$.
\end{prop}

\begin{proof}
\begin{equation*}
0 = D\hh^2 = \tau \hh - \hh \tau. \qedhere 
\end{equation*}
\end{proof}

Next, let $\AP$ be the differential graded algebra 
having a single generator ${\xx}$ of degree $-1$, subject to
\begin{align}
D{\xx} + {\xx}^2 &=0 .
\end{align}
The canonical isomorphisms
$\varepsilon \colon \AH_0 \to R$ and
$\varepsilon \colon \AP_0 \to R$
turn $\AH$ and $\AP$ into augmented differential graded algebras.
Let $\AAA$ denote
the augmented free product differential graded algebra $\AP * \AH$, cf.
\cite{MR0224678}.
We refer to $\AAA=\AP * \AH$ as the 
{\em pseudo perturbation algebra\/}.

Here is an explicit description of that free product:
For two chain complexes $U$ and $V$, let $\TT^n(U,V)$ denote the
chain complex which arises as an $n$-fold tensor product
by alternatingly juxtaposing $U$ and $V$, starting 
with $U$, that is,
\begin{equation}
\TT^n(U,V) = U \otimes V \otimes ... \quad (n \ \text{factors}).
\end{equation}
We use the notation $I$ for the augmentation ideal functor.
As a chain complex, the pseudo perturbation algebra
$\AAA=\AP * \AH$ decomposes as
\begin{equation}
\begin{aligned}
\AP * \AH &= R \oplus \bigoplus_{n \geq 1} \TT^n(I\AP, I\AH)
\oplus \bigoplus_{n \geq 1} \TT^n(I\AH, I\AP)
\\
&=
R \oplus I\AP \oplus I\AH
\oplus \bigoplus_{n \geq 2} \TT^n(I\AP, I\AH)
\oplus \bigoplus_{n \geq 2} \TT^n(I\AH, I\AP) ,
\end{aligned}
\end{equation}
cf.
\cite{MR0224678}.

\section{Pseudo Maurer-Cartan perturbation algebra}
Let  $u =\hh {\xx}$ and  $v ={\xx}\hh$.
We also use the notation $t = 1- \tau$.
The pseudo perturbation algebra 
$\AAA=\AP * \AH$ has  as well ${\xx}$, $\hh$, and $t$ as algebra
generators.
Let $\widehat \AAA$ denote the graded $R$-algebra 
that arises by formally inverting the members $1+u = 1+\hh {\xx}$ and 
$1+v = 1+{\xx}\hh$ of $\AAA_0$.
The differential $D$ of $\AAA$ extends to a differential
on $\widehat \AAA$; we maintain the notation $D$ for this differential.
We refer to $\widehat \AAA$ as the 
{\em pseudo Maurer-Cartan perturbation algebra\/}.

Inspection shows that
\begin{align}
(1+{\xx} \hh)^{-1}&=1-{\xx}(1+ \hh {\xx})^{-1} \hh
\label{insp1}
\\
(1+ \hh {\xx} )^{-1}&=1- \hh (1+ {\xx} \hh)^{-1}{\xx},
\label{insp2}
\end{align}
cf. \cite[Remark 2.4]{MR3276839}.
Below we use the notation 
\begin{equation}
\alpha = (1+u)^{-1} = (1+\hh {\xx})^{-1}\in \widehat \AAA_0,\ 
\beta = (1+v)^{-1} = (1+{\xx}\hh)^{-1} \in \widehat \AAA_0.
\end{equation}
In terms of this notation, \eqref{insp1} and \eqref{insp2}
take the form
\begin{align}
\beta + {\xx} \alpha \hh &=1
\label{insp3}
\\
\alpha + \hh \beta {\xx} &=1.
\label{insp4}
\end{align}

\begin{prop}
Setting 
\begin{equation}
\phi({\xx})=-{\xx},\ \phi(\hh) = \alpha s = s\beta,\ \phi(t) =\alpha t \beta,
\end{equation}
yields
an involution $\phi \colon \widehat \AAA \to \widehat \AAA$
of the graded $R$-algebra $\widehat \AAA$ such that
\begin{equation}
\phi(\alpha)= \alpha^{-1},\ \phi(\beta)= \beta^{-1}.
\end{equation}
\end{prop}

Under the involution $\phi$ of $\widehat \AAA$, the algebra
differential $D$ passes
to the algebra differential $D^\phi= \phi D \phi$ on $\widehat \AAA$.

\begin{lem}
\begin{align}
D\alpha &=  - \alpha (\tau {\xx} +\hh {\xx}^2)\alpha
\label{dif1}
\\
D\beta &=   \beta ({\xx} \tau +{\xx}^2 \hh)\beta 
\label{dif2}
\\
{\xx} \alpha &=\beta {\xx} 
\label{comm}
\\
\alpha \hh &= \hh \beta  .
\label{comm2}
\end{align}
\end{lem}

\begin{proof}
The identities $0 = D(\alpha \alpha^{-1})$ 
and  $0 = D(\beta \beta^{-1})$ 
entail
\begin{align*}
D\alpha &= - \alpha (D\alpha^{-1})\alpha  = - \alpha (D(1+ \hh {\xx}))\alpha
 = - \alpha (\tau {\xx} +\hh {\xx}^2)\alpha
\\
D\beta &= - \beta (D\beta^{-1})\beta  = - \beta (D(1+ {\xx}\hh))\beta
 =  \beta ({\xx} \tau +{\xx}^2 \hh)\beta .
\end{align*}
Further,
\begin{equation*}
{\xx} \alpha - \beta {\xx} ={\xx} - {\xx} \hh {\xx} + {\xx} \hh {\xx} \hh {\xx} - \ldots 
-({\xx} - {\xx} \hh {\xx} + {\xx} \hh {\xx} \hh {\xx} - \ldots )
=0. \qedhere
\end{equation*}

\end{proof}

On $\AAA$,
the member ${\xx}$ of
$\AAA$
induces, in the standard manner, a twisted (or perturbed) differential
$D^{\xx}$. We recall that $D^{\xx}(a) = Da + [{\xx},a]$ ($a \in \AAA$).
This differential turns 
$\AAA$  into 
a differential graded algebra as well, and the twisted differential plainly
extends to $\widehat \AAA$.
We denote the perturbed differential graded algebras by 
$\AAA^{\xx}$ and $\widehat \AAA^{\xx}$.

\begin{thm}
\label{structural}
The algebra differential $D^\phi$ on 
$\widehat \AAA$ coincides with the twisted differential $D^{\xx}$.
\end{thm}

\begin{proof} Using $Dt=0$, $\beta \phi(\beta) =1$, $\phi(\alpha) \alpha = 1$,
$ x \alpha = \beta x$, $\alpha = (1+\hh x)^{-1}$,
$\phi(\alpha) = \alpha^{-1}=1+\hh x$,
$\beta = (1+x \hh)^{-1}$,
and $\phi(\beta) = \beta^{-1}=1+ x\hh$,
we find:
\begin{align*}
D^\phi (\hh)&=\phi (D(\alpha \hh))
\\
&
=\phi (D(\alpha) \hh + \alpha D \hh)
\\
&=\phi (D(\alpha)) \phi(\hh) + \phi( \alpha \tau)
\\
&=\phi (D(\alpha)) \alpha \hh + \phi( \alpha) \phi(1-t)
\\
&=\phi (D(\alpha)) \alpha \hh + \alpha^{-1}(1- \alpha t \beta)
\\
&=\phi (D(\alpha)) \alpha \hh + \alpha^{-1}-  t \beta
\\
&=\phi (  - \alpha (\tau {\xx} +\hh {\xx}^2)\alpha ) \alpha \hh + \alpha^{-1}-  t \beta
\\
&=-\phi (\alpha) (\phi(\tau) \phi({\xx}) +\phi(\hh) \phi({\xx}^2))\phi(\alpha ) 
\alpha \hh 
+ \alpha^{-1}-  t \beta
\\
&=-\phi (\alpha) (\phi(1-t) (-{\xx}) +\alpha \hh {\xx}^2)  \hh 
+ \phi(\alpha)-  t \beta
\\
&=\phi (\alpha) (\phi(1-t) {\xx})\hh 
-\phi (\alpha)
\alpha \hh {\xx}^2  \hh 
+ \phi(\alpha)-  t \beta
\\
&=\phi (\alpha) (1- \phi (t)) {\xx}\hh 
- \hh {\xx}^2  \hh 
+ \phi(\alpha)-  t \beta
\\
&=\phi (\alpha) {\xx}\hh 
-
\phi (\alpha) \phi (t) {\xx}\hh 
- \hh {\xx}^2  \hh 
+ \phi(\alpha)-  t \beta
\\
&=\phi (\alpha) {\xx}\hh 
-
\phi (\alpha) \alpha t \beta {\xx}\hh 
- \hh {\xx}^2  \hh 
+ \phi(\alpha)-  t \beta
\\
&=\phi (\alpha) {\xx}\hh 
-
 t \beta {\xx}\hh 
- \hh {\xx}^2  \hh 
+ \phi(\alpha)-  t \beta
\\
&=(1+\hh {\xx}) {\xx}\hh 
-
 t \beta {\xx}\hh 
- \hh {\xx}^2  \hh 
+ 1+\hh {\xx} -  t \beta
\\
&= {\xx}\hh 
-
 t \beta {\xx}\hh 
+ 1+\hh {\xx} -  t \beta
\\
&= 1 + [{\xx},\hh] 
-
t \beta {\xx}\hh 
 -  t \beta
\\
&= 1 + [{\xx},\hh] 
-
t  {\xx}\hh \beta
 -  t \beta
\\
&= 1 + [{\xx},\hh] 
-
t (1+ {\xx}\hh) \beta
\\
&= 1 + [{\xx},\hh] 
-
t (1+ {\xx}\hh) (1+ {\xx}\hh)^{-1}
\\
&= 1 -t + [{\xx},\hh] 
\\&
= \tau + [{\xx},\hh]
\\&
= D^{\xx} (\hh) 
\end{align*}
Likewise
\begin{align*}
D^\phi ({\xx})&= \phi (D(-{\xx})) 
\\
&= \phi({\xx}^2) = {\xx}^2
\\
D^{\xx} ({\xx})&= D{\xx} +[{\xx},{\xx}] 
\\
&= -{\xx}^2 + 2 {\xx}^2 = {\xx}^2
\end{align*}
Finally,
\begin{align*}
D^\phi(t) &=\phi(D(\phi (t)))
=
\phi(D(\alpha t \beta))
\\
&=
\phi(D(\alpha) t \beta)
+\phi(\alpha  D(t) \beta)
+
\phi(\alpha  t  D(\beta))
\\
&=
\phi(D(\alpha)) \phi(t) \phi(\beta)
+\phi(\alpha)  \phi(D(t)) \phi(\beta)
+
\phi(\alpha)  \phi(t)  \phi(D(\beta))
\\
&=
\phi(D(\alpha)) \phi(t) \phi(\beta)
+
\phi(\alpha)  \phi(t)  \phi(D(\beta))
\\
&=
\phi(- \alpha (\tau {\xx} +\hh {\xx}^2)\alpha) \alpha t \beta \phi(\beta)
+
\phi(\alpha)  \alpha t \beta  \phi( \beta ({\xx}\tau  +{\xx}^2 \hh)\beta))
\\
&=
\phi(- \alpha (\tau {\xx} +\hh {\xx}^2)\alpha) \alpha t 
+
\phi(\alpha)  \alpha t  \phi(({\xx}\tau  +{\xx}^2 \hh)\beta)
\\
&=
\phi(- \alpha (\tau {\xx} +\hh {\xx}^2)) t 
+
 t  \phi(({\xx}\tau  +{\xx}^2 \hh)\beta)
\\
&=
-\phi(\alpha \tau {\xx}) t 
-
\phi(\alpha \hh {\xx}^2) t 
+
 t  \phi({\xx}\tau\beta)
+
 t  \phi({\xx}^2 \hh\beta)
\\
&=
-\phi(\alpha) \phi(\tau) \phi({\xx}) t 
-
\phi(\alpha) \phi(\hh) \phi({\xx})^2t 
+
 t  \phi({\xx})\phi(\tau)\phi(\beta)
+
 t  \phi({\xx})^2 \phi(\hh) \phi(\beta)
\\
&=
\phi(\alpha) (1-\phi(t)) {\xx} t 
-
\phi(\alpha) \alpha \hh {\xx}^2t 
- t  {\xx} (1-\phi(t))\phi(\beta)
+
 t  {\xx}^2 \hh \beta \phi(\beta)
\\
&=
\phi(\alpha)  {\xx} t 
-
\phi(\alpha) \phi(t) {\xx} t 
-
\hh {\xx}^2t 
- t  {\xx} \phi(\beta)
+ t  {\xx} \phi(t)\phi(\beta)
+
 t  {\xx}^2 \hh 
\\
&=
(1+\hh {\xx})  {\xx} t 
-
\phi(\alpha) \alpha t \beta {\xx} t 
-
\hh {\xx}^2t 
- t  {\xx} (1+ {\xx}\hh )
+ t  {\xx} \alpha t \beta\phi(\beta)
+
 t  {\xx}^2 \hh 
 \\
&=
{\xx}t+ \hh {\xx}^2 t  
-
 t \beta {\xx} t 
-
\hh {\xx}^2t 
- t  {\xx} 
- t  {\xx}^2\hh
+ t  {\xx} \alpha t 
+
 t  {\xx}^2 \hh 
\\
&=
{\xx}t 
-
 t \beta {\xx} t  
- t  {\xx} 
+ t  {\xx} \alpha t 
\\
&=
[{\xx},t] 
+ t  ({\xx} \alpha - \beta {\xx})t 
=
[{\xx},t] 
=
D^{\xx}(t) \qedhere. 
\end{align*}
\end{proof}

\section{Pseudo perturbation lemma}

From the introduction, we recall that a {\em pseudocontraction\/}
consists of a chain complex $\NNN$, together with
a chain endomorphism $\tau \colon \NNN \to \NNN$
and a homogeneous degree $1$ operator
$\hhh \colon \NNN \to \NNN$, subject to,
with $\hhh$ substituted for $\hh$, \eqref{pc1} and {\eqref{pc2}.
Pseudocontractions manifestly correspond bijectively
to differential graded $\AH$-modules.
A pseudocontraction  $(\NNN,\tau, \hhh)$ 
having $\tau = \NNN$
is an ordinary cone,
together with a conical contraction, cf., e.g., 
\cite[IV.1.5 p.~168]{MR0365571} for this notion.
This observation justifies, perhaps, our pseudocontraction terminology.
In Proposition \ref{compar} we spell out the relationship between
pseudocontractions and ordinary contractions.

Consider a pseudocontraction $(\NNN,\tau,\hhh)$.
Recall that a {\em perturbation\/}
$\partial$ of the differential $d$
on $\NNN$ is a homogeneous degree $-1$ operator $\partial$ on 
$\NNN$ such that the operator $d+ \partial$ on $\NNN$ has square zero,
i.e., is itself a differential.
The pseudocontraction structure
$(\hhh,\tau)$ on $\NNN$
being equivalent to an $\AH$-module structure on $\NNN$
over the pseudocontraction algebra $\AH$,  the perturbation
$\partial$ determines and is determined by a unique
extension to an
$\AAA$-module structure on $\NNN$
over the pseudo perturbation algebra $\AAA =\AP  *\AH$.
Henceforth our convention is this:
We distinguish in notation between $\hh,x \in \AAA$ and the operators 
$\hhh$ and $\partial$ 
on $\NNN$ they determine, but we do not distinguish 
in notation
between $t,\tau, \alpha, \beta \in \widehat \AAA$ 
and the operators 
they determine 
on $\NNN$ (provided that
the degree zero endomorphisms 
$\NNN + h \partial $ and $\NNN + \partial h$ of $\NNN$
are invertible).

Let $\NNN_\partial$ denote the chain complex
$(N,d + \partial)$, and 
write 
\begin{align}
t_\partial &= \alpha t \beta \colon \NNN \to \NNN
\\
h_{\partial}&=\alpha h
=h\beta  \colon \NNN \to \NNN.
\label{1.1.4.0}
\end{align}

\begin{cor}[Pseudo perturbation lemma]
\label{pseudolem}
Suppose that the degree zero endomorphisms 
$\NNN + h \partial $ and $\NNN + \partial h$ of $\NNN$
are invertible, that is, that the
$\AAA$-module structure on $\NNN$ extends to an
$\widehat \AAA$-module structure 
on $\NNN$
over the
pseudo Maurer Cartan perturbation algebra $\widehat \AAA$.
Then
$(\NNN_\partial, \NNN -t_\partial , \hhh_\partial)$
is a pseudocontraction as well.
\end{cor}

\begin{proof}
The chain complex $\NNN_\partial$
is a module over $\widehat \AAA^{\xx}$. 
The composite
$\AH \stackrel{\subseteq}\longrightarrow 
\widehat A \stackrel{\phi}\longrightarrow 
\widehat A^{\xx}$
turns  $\NNN_\partial$ into an $\AH$-module in such a way
that the members $\tau$ and $\hh$ act on $\NNN$ as
the operators
$\NNN -t_\partial$ and $\hhh_\partial$. This establishes the
assertion since
$\AH$-module structures characterize pseudocontractions.
\end{proof}

\begin{rema}{\rm 
Suppose that $\NNN$ is a filtered chain complex, that the
filtration is complete, see, e.g., \cite[VIII.8 p.~292]{MR0346025}, and
let $\partial$ be a perturbation of the differential $d$ of $\NNN$
that lowers filtration. 
Then the series
$\sum_{n\geq 0}  (-h\partial)^n$
and
$\sum_{n\geq 0}(-\partial h)^n$
converge, and hence
the degree zero endomorphisms 
$\NNN + h \partial $ and $\NNN + \partial h$ of $\NNN$
are invertible.
In practice, for the degree filtration of a chain complex
that is bounded below (e.g., concentrated in non-negative degrees),
completeness is immediate. 
In fact, the convergence is then naive in the sense that, 
evaluated on a specific homogeneous element, 
$\sum_{n\geq 0}  (-h\partial)^n$
and
$\sum_{n\geq 0}(-\partial h)^n$ yield finite sums.
}
\end{rema}

Define a
{\em weak contraction\/}
 $  (\MMM
     \begin{CD}
      \null @>{\nabla}>> \null\\[-3.2ex]
      \null @<<{\pi}< \null
     \end{CD}
    \NNN, h) 
$
of chain complexes to consist of

-- chain complexes $\MMM$ and $\NNN$,
\newline
\indent -- a surjective chain map $\pi\colon \NNN \to \MMM$ and 
an injective chain map $\nabla \colon \MMM \to
\NNN$,
\newline
\indent --  a morphism $h\colon \NNN \to \NNN$ of the underlying graded
modules of degree 1,
\newline
subject  to the axioms
\begin{align}
Dh &= \NNN -\nabla \pi, \label{co1}
\\
hh &= 0 .
\label{side1}
\end{align}

Given a pseudocontraction $(\NNN,\tau, \hhh)$,
let
$
\MMM= t \NNN \subseteq \NNN
$,
let $\pi =t\colon \NNN \to \MMM$, and denote the
injection $\MMM \subseteq \NNN$ by $\nabla \colon \MMM \to \NNN$.
Since $t$ is a chain map,
$\MMM$ is a chain complex,
$\pi$ and $\nabla$ are chain maps, and
 $  (\MMM
     \begin{CD}
      \null @>{\nabla}>> \null\\[-3.2ex]
      \null @<<{\pi}< \null
     \end{CD}
    \NNN, h) 
$
is a weak contraction.
Further,
$t = \nabla \pi$.
Likewise, a weak contraction
$  (\MMM
     \begin{CD}
      \null @>{\nabla}>> \null\\[-3.2ex]
      \null @<<{\pi}< \null
     \end{CD}
    \NNN, h) 
$
determines the 
pseudocontraction ${(\NNN, \NNN- \nabla \pi, \hhh)}$.
In this vein, 
the assignment to $(\NNN,\tau, \hhh)$ of  $  (\MMM
     \begin{CD}
      \null @>{\nabla}>> \null\\[-3.2ex]
      \null @<<{\pi}< \null
     \end{CD}
    \NNN, h) 
$
yields an equivalence between
pseudocontractions
and weak contractions.

Consider a weak contraction
$  (\MMM
     \begin{CD}
      \null @>{\nabla}>> \null\\[-3.2ex]
      \null @<<{\pi}< \null
     \end{CD}
    \NNN, h) $.
Let $\partial$ be a perturbation of the differential on $\NNN$,
and
suppose that the degree zero endomorphisms 
$\NNN + h \partial $ and $\NNN + \partial h$ of $\NNN$
are invertible.
Let
\begin{align}
\dell &= 
\ppi\partial \alpha\nabla 
=\ppi\beta\partial \nabla \colon \MMM \to \MMM
\label{1.1.1.0}
\\
\nabla_{\partial}&= 
\alpha\nabla \colon \MMM \to \NNN
\label{1.1.2.0}
\\
\ppi_{\partial}&
=\ppi\beta \colon \NNN \to \MMM ,
\label{1.1.3.0}
\end{align}
and let $\MMM_\dell$ denote the graded object $\MMM$, endowed with the 
operator $d+\dell$.
Plainly,
\begin{equation}
t_\partial \, (= \alpha t \beta) = \nabla_\partial \pi_\partial.
\label{plainly}
\end{equation}

\begin{lem} 
\label{tech}
The operator $\dell$ on $\MMM$ satisfies the identities
\begin{align}
\pi_\partial (d+ \partial) &= (d+\dell) \pi_\partial
\label{tech1}
\\
 (d+ \partial) \nabla_\partial&= \nabla_\partial(d+\dell) .
\label{tech2}
\end{align}
Hence $\dell$ is a perturbation of the differential on $\MMM$,
and 
$\pi_\partial \colon \NNN_\partial \to \MMM_\dell$ and
$\nabla_\partial \colon \MMM_\dell \to \NNN_\partial$
are chain maps. Furthermore,
\begin{align}
\pi_\partial\nabla_\partial(d+\dell) =
\pi_\partial (d+ \partial)\nabla_\partial &= (d+\dell) \pi_\partial\nabla_\partial .
\end{align}
\end{lem}

\begin{proof}
Identity 
 \eqref{dif2} entails
$D\beta =   \beta ({\xxx} \tau +{\xxx}^2 \hhh)\beta$.
Hence
\begin{align*}
\pi_\partial \circ(d + \partial)&=\pi \beta \circ(d + \partial)
\\
&= \pi \beta d  + \pi \beta \partial
\\
&=\pi (d \beta - \beta ({\xxx} \tau +{\xxx}^2 \hhh)\beta) + \pi \beta \partial
\\
&=\pi d \beta -  \pi \beta  {\xxx} (\tau\beta +  {\xxx} \hhh \beta) + \pi \beta \partial
\\
&=d \pi \beta  -  \pi \beta  {\xxx} ((1-t)\beta +  {\xxx} \hhh \beta) + \pi \beta \partial
\\
&=d \pi \beta  
-  \pi \beta  {\xxx} (1+  {\xxx} \hhh) \beta 
+\pi \beta  {\xxx} t\beta 
+ \pi \beta \partial
\\
&=d \pi \beta  
+\pi \beta  {\xxx} t\beta 
\\
&=d \pi \beta  
+\pi \beta  \partial  \nabla \pi \beta 
\\
&
= (d+ \pi \beta \partial \nabla)\circ \pi \beta
\\
&
= (d+ \pi \partial \alpha \nabla)\circ \pi \beta
\\
&
= (d+ \dell)\circ \pi_\partial .
\end{align*}
Likewise,
identity \eqref{dif1} entails  
$D\alpha =  - \alpha (\tau {\xxx} +\hhh {\xxx}^2)\alpha$.
Hence
\begin{align*}
 (d+ \partial) \nabla_\partial&=  (d+ \partial) \alpha \nabla
\\
&= d \alpha \nabla + \partial \alpha \nabla
\\
&=  (\alpha d -\alpha (\tau {\xxx} +\hhh {\xxx}^2)\alpha) \nabla + \partial \alpha \nabla
\\
&=  \alpha d \nabla  -\alpha ((1-t)  +\hhh {\xxx}){\xxx}\alpha \nabla + \partial \alpha \nabla
\\
&=  \alpha \nabla d  -\alpha (1  +\hhh {\xxx}){\xxx}\alpha \nabla 
+\alpha t {\xxx}\alpha \nabla
+ \partial \alpha \nabla
\\
&=  \alpha \nabla d   
+\alpha t {\xxx}\alpha \nabla
\\
&=  \alpha \nabla d   
+\alpha \nabla \pi {\xxx}\alpha \nabla
\\
&=  \alpha \nabla 
(d   + \pi \partial\alpha \nabla)
\\
&=  \nabla_\partial 
(d   + \dell) . \qedhere
\end{align*}
\end{proof}

\begin{cor}[Pseudo perturbation lemma; second version]
\label{pseudolem2}
Let
$  (\MMM
     \begin{CD}
      \null @>{\nabla}>> \null\\[-3.2ex]
      \null @<<{\pi}< \null
     \end{CD}
    \NNN, h) 
$
be a weak contraction of chain complexes,
let $\partial$ be a perturbation of the differential on $\NNN$,
and
suppose that the degree zero endomorphisms 
$\NNN + h \partial $ and $\NNN + \partial h$ of $\NNN$
are invertible.
Then
\begin{equation}
   \left(M_{\dell}
     \begin{CD}
      \null @>{\nabla_{\partial}}>> \null\\[-3.2ex]
     \null @<<{\ppi_{\partial}}< \null
     \end{CD}
   N_{\partial} , h_{\partial} \right)
\label{2.66w}
  \end{equation}
is a weak a contraction.
\end{cor}

\begin{proof} The weak contraction
 $  (\MMM
     \begin{CD}
      \null @>{\nabla}>> \null\\[-3.2ex]
      \null @<<{\pi}< \null
     \end{CD}
    \NNN, h) 
$
determines the pseudocontraction 
\[
(\NNN,\tau,\hhh) = (\NNN,\NNN - \nabla \pi,\hhh),
\]
and the pseudocontraction structure
and the perturbation $\partial$ determine an
$\widehat \AAA$-module structure
on $\NNN$ over the pseudo Maurer-Cartan perturbation algebra
$\AAA= \AP * \AH$. By Corollary \ref{pseudolem},
$(\NNN_\partial, \NNN-t_\partial, \hhh_\partial)$
is a pseudocontraction, that is,
\begin{align*}
\hhh_\partial^2 &=0
\\
(d + \partial)\circ t_\partial &= t_\partial \circ (d + \partial)
\\
(d + \partial)\circ \hhh_\partial 
+ \hhh_\partial \circ (d + \partial) &= \NNN -t_\partial  
\\
&=\NNN -\nabla_\partial \pi_\partial ,
\end{align*}
cf. \eqref{plainly} above.
In view of Lemma \ref{tech},
we conclude that \eqref{2.66w} is a weak contraction.
\end{proof}

\begin{rema}
\label{content}
{\rm
Under the circumstances of Corollary \ref{pseudolem2},
the perturbed pseudocontraction
$(\NNN_\partial,\NNN - t_\partial, \hhh_\partial)$
determines the weak contraction
 $  \left(\left(t_\partial\NNN, (d + \partial)|_{t_\partial \NNN}\right)
     \begin{CD}
      \null @>{j}>> \null\\[-3.2ex]
      \null @<<{t_\partial}< \null
     \end{CD}
    \NNN_\partial, h_\partial\right) 
$.
Inspection of the diagram
\begin{equation*}
\xymatrixcolsep{4.5pc}
\xymatrix{
\NNN \ar[r]^\pi
&
\MMM \ar[dr]_{\nabla_\partial} \ar[r]^\nabla
&
\NNN
\ar[d]^\alpha
\\
\NNN \ar[u]^\beta \ar[r]_{t_\partial}
&
t_\partial \NNN  \ar[r]_j
&
\NNN
}
\end{equation*}
shows that the values of $\nabla_\partial = \alpha \nabla$
lie in $t_\partial \NNN$ in such a way that 
$\nabla_\partial$ is 
chain isomorphism 
\begin{equation}
\nabla_\partial\colon \MMM_\dell=(\MMM, d + \dell) \longrightarrow 
(t_\partial \NNN, (d + \partial)|_{t_\partial \NNN}).
\label{tech3}
\end{equation} 
The morphism $\nabla_\partial$ being a chain map of the kind \eqref{tech3}
is the content of identity \eqref{tech2}.

}
\end{rema}

\section{Relationship with ordinary homological perturbation theory}
\label{hpt}

The reader can find details about
H(omological) P(erturbation) T(heory)
in
\cite{MR2640649, MR2762544, MR2762538, MR2820385, MR1109665, MR1932522}.
Among the classical references are
\cite{MR0220273, MR0056295, MR0065162, MR0301736, MR662761}.

A
{\em contraction\/} of chain complexes is a weak contraction
 $  (\MMM
     \begin{CD}
      \null @>{\nabla}>> \null\\[-3.2ex]
      \null @<<{\pi}< \null
     \end{CD}
    \NNN, h) 
$
 subject to, furthermore,
the axioms
\begin{align}
 \pi \nabla &= \MMM,
\label{co0}
\\
\pi h &= 0, \quad h \nabla = 0 \qquad
\text{({\em annihilation
properties\/} or {\em side conditions\/})}. 
\label{side}
\end{align}
\begin{rema}{\rm
In the definition of a contraction, as opposed to that of a weak contraction,
there is no need to require $\pi$ to be surjective and $\nabla$ to be
injective since these properties are consequences of \eqref{co0}.
}
\end{rema}

For a contraction of chain complexes of the particular kind
$
\Nsddata {\NNN} {\phantom a\nabla}{\phantom a\pi}
{\HH(\NNN)}h $,
letting
$
\HHH = \ker (h) \cap \ker(d) = \nabla \HH(\NNN)$,
we see that the homogeneous degree $j$ constituent $\NNN_j$ ($j \in \mathbb Z$) 
of $\NNN$ decomposes as
\begin{equation} 
\NNN_j = d \NNN_{j+1} \oplus \HHH_j \oplus h(d\NNN_j). 
\label{4.5.1}
\end{equation}
In the situation of Example \ref{KS} below,
\eqref{4.5.1} plays the role of a {\em Hodge decomposition\/}.
On p.~19 
of \cite{MR0195995},
Nijenhuis and Richardson indeed
refer to a decomposition of the kind \eqref{4.5.1}
(not using the language of homological perturbation theory)
as a \lq\lq Hodge decomposition\rq\rq.

\begin{examp}[Kodaira-Spencer Lie algebra] 
\label{KS}
{\rm See \cite{MR0112157, MR0112154}. Take the ground ring 
to be the field $\CC$
of complex numbers,  
consider a complex manifold $M$, 
let $\tau_M$
denote the holomorphic tangent bundle
of $M$, let $\overline \partial$ be the corresponding
Dolbeault operator,
and let
$\gg = (\AAA^{(0,*)}(M,\tau_M), \overline \partial)$
be the {\em Kodaira-Spencer algebra\/} of $M$,
endowed with the homological grading 
\begin{equation}
\gg_0 = \AAA^{(0,0)}(M,\tau_M),
\quad
\gg_{-1} = \AAA^{(0,1)}(M,\tau_M),
\quad
\gg_{-2} = \AAA^{(0,2)}(M,\tau_M),
\quad
\text{etc.}
\end{equation}
Thus, with our convention on degrees,
$\HH_*(\gg) = \HH^{-*}(M,\tau_M)$,
the
cohomology of $M$ with values in the sheaf of germs of
holomorphic vector fields.
A Hodge decomposition of $\gg$ now yields a special kind of contraction.
}
\end{examp}

Following \cite{chuanglazarev}, define an {\em abstract Hodge decomposition\/}
of a chain complex $X$ to consist of operators
$t$ and $h$ on $X$ of degree $0$ and $1$, respectively, such that
\begin{align}
h^2 &= 0
\\
Dh &= 1 -t
\\
Dt&=0
\\
t^2 &= t
\label{ah4}
\\
th = ht &=0.
\label{ah5}
\end{align}

\begin{rema} The conditions characterizing an abstract Hodge decomposition
are not independent. For example,
$ht=0$ implies $t^2=t$:
$0= D(ht) = (Dh)t =(1-t)t$.
\end{rema}

An abstract Hodge decomposition is a special kind of  pseudocontraction,
and contractions and abstract Hodge decompositions
are equivalent notions:
Let
$
\Nsddata {\NNN} {\phantom a\nabla}{\phantom a\pi}
{\MMM}h 
$
be a contraction of chain complexes, and 
let $t = \nabla \pi$. Then $t$ and $h$ yield an abstract Hodge decomposition
of $\NNN$. Likewise, let  $(\NNN, \tau,\hhh)$ be a pseudocontraction,
let $t = \mathrm{Id}- \tau \colon \NNN \to \NNN$,
let $\MMM = t \NNN$, and let $j \colon M \to \NNN$ denote the inclusion.

\begin{prop}
\label{compar}
Let  $(\NNN, \tau,\hhh)$ be a pseudocontraction.
The following are equivalent.
\begin{enumerate}
\item[{\rm (i)}]
The operators $\hhh$ and $t=1-\tau$
yield an abstract Hodge decomposition of $\NNN$. 
\item[{\rm (ii)}]
The operators
$\hhh$ and $t=1-\tau$ satisfy {\rm \eqref{ah4}} and {\rm \eqref{ah5}}.
\item[{\rm (iii)}]
Beyond the side condition $h^2=0$, 
the operators
$\hhh$ and $t=1-\tau$ satisfy
the side conditions $t h=0$ and $h j=0$, cf. {\rm \eqref{side}}, that is,
$
\Nsddata {\NNN} {\phantom a j}{\phantom a t}
{\MMM}h 
$
is an ordinary  contraction. 
\end{enumerate}

\end{prop}

\begin{proof}
This is straightforward. We only note that
\eqref{ah4} is equivalent to \eqref{co0}.
\end{proof}

\begin{cor}[Ordinary perturbation lemma]
\label{olem} 
Let
$  (\MMM
     \begin{CD}
      \null @>{\nabla}>> \null\\[-3.2ex]
      \null @<<{\pi}< \null
     \end{CD}
    \NNN, h) 
$
be a contraction of chain complexes,
let $\partial$ be a perturbation of the differential on $\NNN$,
and
suppose that the degree zero endomorphisms 
$\NNN + h \partial $ and $\NNN + \partial h$ of $\NNN$
are invertible.
Then
\begin{equation}
   \left(M_{\dell}
     \begin{CD}
      \null @>{\nabla_{\partial}}>> \null\\[-3.2ex]
     \null @<<{\ppi_{\partial}}< \null
     \end{CD}
   N_{\partial} , h_{\partial} \right)
\label{2.66}
  \end{equation}
constitutes a contraction.

\end{cor}

\begin{rema}{\rm
Writing out \eqref{1.1.4.0} and \eqref{1.1.1.0} -- \eqref{1.1.3.0}
explicitly yields
the standard expressions
in
the perturbation lemma, see, e.g.,
\cite[Lemma 9.1]{MR2640649}:
\begin{align*} \dell &
=\ppi\partial (1+ h\partial)^{-1}\nabla = \sum_{n\geq 0} \ppi\partial (-h\partial)^n\nabla 
\\
&=\ppi(1+ \partial h)^{-1}\partial \nabla 
=
\sum_{n\geq 0} \ppi(-\partial h)^n\partial\nabla
\\
\nabla_{\partial}&
=(1+ h\partial)^{-1}\nabla = \sum_{n\geq 0} (-h\partial)^n\nabla
\\
\ppi_{\partial}&
=\ppi(1+ \partial h)^{-1}= \sum_{n\geq 0} \ppi(-\partial h)^n
\\
h_{\partial}&=(1+ h\partial)^{-1} h=\sum_{n\geq 0} (-h\partial)^n h 
\\
&=h(1+ \partial h)^{-1} 
=\sum_{n\geq 0}
h(-\partial h)^n
\end{align*}
}
\end{rema}

\begin{proof} In view of Corollary \ref{pseudolem2}, it remains to confirm \eqref{co0} and \eqref{side} for the 
perturbed data, that is, we must show that
$\pi_\partial \nabla_\partial= \MMM$ and
$\pi_\partial \hhh_\partial=0=
 \hhh_\partial\nabla_\partial$.
Using \eqref{co0} and \eqref{side} for the 
unperturbed data, we find
\begin{align*}
\pi_\partial \nabla_\partial &=\pi \beta \alpha \nabla
\\
&=
\pi (1+x \hh)^{-1}(1+ \hh x)^{-1}\nabla
\\
&=
\pi \sum_{n\geq 0} (-\partial h)^n\sum_{n\geq 0} (-h\partial)^n\nabla
\\
&=
\pi (1 -\partial h -h\partial +(\partial h)^2 + \partial h h\partial
+h\partial \partial h +(h \partial)^2 + \ldots)
\nabla
\\
&= \pi \nabla
\\
&=\MMM .
\end{align*}
The same kind of reasoning shows that
$\pi_\partial \hhh_\partial=0=
 \hhh_\partial\nabla_\partial$.
\end{proof}

\begin{rema}{\rm
Chuang-Lazarev refer to \cite[Theorem 3.5]{chuanglazarev} as the
\lq\lq abstract version of the HPL\rq\rq\ (homological perturbation lemma)
and claim that the \lq\lq ordinary HPL is a consequence of the
abstract one\rq\rq. They spell out this consequence
as \cite[Corollary 3.7]{chuanglazarev}.
\cite[Theorem 3.5]{chuanglazarev}  is similar to Theorem \ref{structural}
above, except that it incorporates the side conditions 
\eqref{side} and \eqref{co0} (or an equivalent  condition),
and \cite[Corollary 3.7]{chuanglazarev} yields
a result similar to Corollary \ref{pseudolem} above, 
but again with the side conditions \eqref{side} 
and a condition of the kind \eqref{co0}
incorporated.
From the resulting perturbed abstract Hodge decomposition of the kind
$(\NNN_\partial, t_\partial, \hhh_\partial)$, we can at once
deduce the 
contraction
\begin{equation}
  \left(\left(t_\partial\NNN, (d + \partial)|_{t_\partial \NNN}\right)
     \begin{CD}
      \null @>{j}>> \null\\[-3.2ex]
      \null @<<{t_\partial}< \null
     \end{CD}
    \NNN_\partial, h_\partial\right) .
\label{atonce}
\end{equation}
However, cf. Remark \ref{content} above,
when we start with a contraction 
 $  (\MMM
     \begin{CD}
      \null @>{\nabla}>> \null\\[-3.2ex]
      \null @<<{\pi}< \null
     \end{CD}
    \NNN, h) 
$
and a perturbation $\partial$ of the differential on $\NNN$,
we cannot deduce, from \eqref{atonce},  the perturbation of the kind $\dell$
of the differential on $\MMM$, cf. \eqref{1.1.1.0},
without further thought. Lemma \ref{tech} provides the requisite
further thought.
}
\end{rema}

\section{Insight into the structure of the pseudo Maurer-Cartan perturbation algebra}

As before, let  $u =\hh {\xx}$ and  $v ={\xx}\hh$.
We use the notation
$p(u,\tau)$, $p_1(u,\tau)$, $p_2(u,\tau)$, etc.
for non-commutative monomials
in $u$ and $\tau$ that
involve $u$ non-trivially
(but do not necessarily involve $\tau$) and
the notation
$q(v,\tau)$, $q_1(v,\tau)$, $q_2(v,\tau)$, etc.
for non-commutative monomials
in $v$ and $\tau$ that
involve $v$ non-trivially
(but do not necessarily involve $\tau$).
Further, we occasionally write  the multiplication
map (product operatioon) of $\AAA$ as 
$\,\cdot\, \colon \AAA \otimes \AAA \to \AAA$.

\begin{prop}
The degree zero algebra
$\AAA_0$ of the graded algebra
$\AAA$ has the following structural properties.
\begin{enumerate}
\item[{\rm(i)}]
As an $R$-module,
$\AAA_0$ is free, 
having as basis the monomials 
in the union of the four families
of the following kind:
\begin{itemize}
\item the monomials in $\tau$,
\item the monomials of the kind $p(u,\tau)$, 
\item the monomials of the kind $p(v,\tau)$, 
\item the monomials of the kind $p(u,\tau)p(v,\tau)$.
\end{itemize}

\item[{\rm(ii)}]
Iuxtaposition realizes products in $\AAA_0$ of the kind
\begin{equation}
\begin{gathered}
p(u,\tau) \cdot \tau^j,\ 
\tau^j\cdot p(u,\tau),\ 
q(v,\tau) \cdot \tau^j,\
\tau^j\cdot q(v,\tau),
\\ 
p(u,\tau) \cdot q(v,\tau),\ 
p_2(u,\tau) \cdot p_1(u,\tau) q(v,\tau),\ 
p(u,\tau)q_1(v,\tau) \cdot q_2(v,\tau).
\end{gathered}
\end{equation}

\item[{\rm(iii)}] Products of the kind
\begin{equation}
q(v,\tau) \cdot p(u,\tau),\ 
 p_1(u,\tau) q(v,\tau) \cdot p_2(u,\tau),\ 
q_2(v,\tau) \cdot p(u,\tau)q_1(v,\tau) 
\end{equation}
are zero.

\item[{\rm(iv)}] Hence, for a monomial
 of the kind $p(u,\tau)p(v,\tau)$,
\begin{equation}
(p(u,\tau)p(v,\tau))^2 = 0 .
\end{equation}

\item[{\rm(v)}]
As an $R$-algebra, $\AAA_0$ has the multiplicative generators
$u$, $v$, and $\tau$, subject to the relations
\begin{align}
v \tau^ju &= 0,\ j \geq 0.
\end{align}

\end{enumerate}
\end{prop}

\begin{proof}
Consider a non-commutative monomial of the kind
\begin{equation}
u^{k_1} v^{\ell_1} \tau^{m_1}u^{k_2} v^{\ell_2} 
\tau^{m_2}\ldots u^{k_a} v^{\ell_a} \tau^{m_a},\ k_j,\ell_j,m_j \geq 0,
\ 0 \leq j \leq a.
\label{monom}
\end{equation}
Suppose that \eqref{monom} is non-zero in $\AAA_0$.
If $\ell_1 = \ldots =\ell_a=0=k_1 = \ldots =k_a$,
\eqref{monom} is a monomial in $\tau$.
Now suppose that
\eqref{monom} is not merely a monomial in $\tau$.
If $\ell_1 = \ldots =\ell_a=0$,
\eqref{monom} is of the kind $p(u,\tau)$.
If $k_1 = \ldots =k_a=0$,
\eqref{monom} is of the kind $q(v,\tau)$.
Suppose that
some $k_i$ and
some $\ell_j$ are non-zero, and let
$\ell_u$ be the smallest member among the non-zero $\ell_j$s.
Then $\ell_1 = \ldots = \ell_{u-1}=0$ and,
since
$v \tau^j u={\xx}\hh \tau^j \hh {\xx}=0 \in \AAA_0$ 
and since \eqref{monom} is non-zero,
we conclude
$k_{u+1}= \ldots = k_a=0$,
that is,
\eqref{monom} is of the kind $q(u,\tau)q(v,\tau)$.
\end{proof}

The homology algebras of the
differential graded algebras $\AH$, $\AP$, and $\AAA$
plainly reduce to isomorphisms
$\varepsilon\colon \Ho(\AH) \to R$,
$\varepsilon\colon \Ho(\AP) \to R$,
$\varepsilon \colon\Ho(\AAA) \to R$. More precisely:

\begin{prop}The differential graded algebras
$\AH$ and $\AP$ admit obvious algebra contractions
\begin{gather}
\Nsddata {\AH} {\phantom a j}{\phantom a\varepsilon}
{R}{h_\AH} 
\label{AHc}
\\
\Nsddata {\AP} {\phantom a j}{\phantom a\varepsilon}
{R}{h_\AP} ,
\label{APc}
\end{gather}
and these contractions induce an algebra contraction
\begin{equation}
\Nsddata {\AAA} {\phantom a j}{\phantom a\varepsilon}
{R}{h_\AAA} .
\label{AAc}
\end{equation}
Furthermore, application of the perturbation lemma
yields an algebra contraction
\begin{equation}
\Nsddata {\AAA^x} {\phantom a j}{\phantom a\varepsilon}
{R}{h_{\AAA^x}} .
\label{AAcx}
\end{equation}
\end{prop}

\begin{proof}
This is straightforward. We leave the details 
to the reader.
\end{proof}

\begin{rema}
{\rm An obvious question is whether the contracting homotopy
$h_\AAA$ in \eqref{AAc} extends to a contracting homotopy for 
the pseudo Maurer-Cartan perturbation algebra
$\widehat \AAA$.

}
\end{rema}

\section*{Acknowledgement}

I am indebted to Jim Stasheff for a number of
most valuable comments on a draft of the paper.
I gratefully acknowledge support by the CNRS and by the
Labex CEMPI (ANR-11-LABX-0007-01).

\bibliographystyle{plain}
\def\cprime{$'$} \def\cprime{$'$} \def\cprime{$'$} \def\cprime{$'$}
  \def\cprime{$'$} \def\cprime{$'$} \def\cprime{$'$} \def\cprime{$'$}
  \def\dbar{\leavevmode\hbox to 0pt{\hskip.2ex \accent"16\hss}d}
  \def\cprime{$'$} \def\cprime{$'$} \def\cprime{$'$} \def\cprime{$'$}
  \def\cprime{$'$} \def\Dbar{\leavevmode\lower.6ex\hbox to 0pt{\hskip-.23ex
  \accent"16\hss}D} \def\cftil#1{\ifmmode\setbox7\hbox{$\accent"5E#1$}\else
  \setbox7\hbox{\accent"5E#1}\penalty 10000\relax\fi\raise 1\ht7
  \hbox{\lower1.15ex\hbox to 1\wd7{\hss\accent"7E\hss}}\penalty 10000
  \hskip-1\wd7\penalty 10000\box7}
  \def\cfudot#1{\ifmmode\setbox7\hbox{$\accent"5E#1$}\else
  \setbox7\hbox{\accent"5E#1}\penalty 10000\relax\fi\raise 1\ht7
  \hbox{\raise.1ex\hbox to 1\wd7{\hss.\hss}}\penalty 10000 \hskip-1\wd7\penalty
  10000\box7} \def\polhk#1{\setbox0=\hbox{#1}{\ooalign{\hidewidth
  \lower1.5ex\hbox{`}\hidewidth\crcr\unhbox0}}}
  \def\polhk#1{\setbox0=\hbox{#1}{\ooalign{\hidewidth
  \lower1.5ex\hbox{`}\hidewidth\crcr\unhbox0}}}
  \def\polhk#1{\setbox0=\hbox{#1}{\ooalign{\hidewidth
  \lower1.5ex\hbox{`}\hidewidth\crcr\unhbox0}}}

%\bibliography{$HOME/johannes/huebschm/refs}

\begin{thebibliography}{10}

\bibitem{MR1057939}
Donald~W. Barnes and Larry~A. Lambe.
\newblock A fixed point approach to homological perturbation theory.
\newblock {\em Proc. Amer. Math. Soc.}, 112(3):881--892, 1991.

\bibitem{MR1802006}
Donald~W. Barnes and Larry~A. Lambe.
\newblock Correction to: ``{A} fixed point approach to homological perturbation
  theory'' [{P}roc. {A}mer. {M}ath. {S}oc. {\bf 112} (1991), no. 3, 881--892;
  {MR}1057939 (91j:55019)].
\newblock {\em Proc. Amer. Math. Soc.}, 129(3):941, 2001.

\bibitem{MR3276839}
Alexander Berglund.
\newblock Homological perturbation theory for algebras over operads.
\newblock {\em Algebr. Geom. Topol.}, 14(5):2511--2548, 2014.

\bibitem{MR0220273}
Ronald Brown.
\newblock The twisted {E}ilenberg-{Z}ilber theorem.
\newblock In {\em Simposio di {T}opologia ({M}essina, 1964)}, pages 33--37.
  Edizioni Oderisi, Gubbio, 1965.

\bibitem{chuanglazarev}
Joseph Chuang and Andrey Lazarev.
\newblock On the perturbation algebra.
\newblock {\em J. Algebra}, 519:130--148, 2019.
\newblock \url{https://arxiv.org/abs/1703.05296}.

\bibitem{MR0056295}
Samuel Eilenberg and Saunders Mac~Lane.
\newblock On the groups of {$H(\Pi,n)$}. {I}.
\newblock {\em Ann. of Math. (2)}, 58:55--106, 1953.

\bibitem{MR0065162}
Samuel Eilenberg and Saunders Mac~Lane.
\newblock On the groups {$H(\Pi,n)$}. {II}. {M}ethods of computation.
\newblock {\em Ann. of Math. (2)}, 60:49--139, 1954.

\bibitem{MR0301736}
Victor K. A.~M. Gugenheim.
\newblock On the chain-complex of a fibration.
\newblock {\em Illinois J. Math.}, 16:398--414, 1972.

\bibitem{MR662761}
Victor K. A.~M. Gugenheim.
\newblock On a perturbation theory for the homology of the loop-space.
\newblock {\em J. Pure Appl. Algebra}, 25(2):197--205, 1982.

\bibitem{MR0346025}
Peter~John Hilton and Urs Stammbach.
\newblock {\em A course in homological algebra}.
\newblock Springer-Verlag, New York, 1971.
\newblock Graduate Texts in Mathematics, Vol. 4.

\bibitem{MR1710565}
Johannes Huebschmann.
\newblock Berika{\v s}vili's functor {$D$} and the deformation equation.
\newblock {\em Proc. A. Razmadze Math. Inst.}, 119:59--72, 1999.
\newblock \url{https://arxiv.org/abs/math/9906032}.

\bibitem{MR2640649}
Johannes Huebschmann.
\newblock On the construction of {$A_\infty$}-structures.
\newblock {\em Georgian Math. J.}, 17(1):161--202, 2010.
\newblock \url{https://arxiv.org/abs/0809.4791}.

\bibitem{MR2762544}
Johannes Huebschmann.
\newblock The {L}ie algebra perturbation lemma.
\newblock In {\em Higher structures in geometry and physics}, volume 287 of
  {\em Progr. Math.}, pages 159--179. Birkh\"auser/Springer, New York, 2011.
\newblock \url{https://arxiv.org/abs/0708.3977}.

\bibitem{MR2762538}
Johannes Huebschmann.
\newblock Origins and breadth of the theory of higher homotopies.
\newblock In {\em Higher structures in geometry and physics}, volume 287 of
  {\em Progr. Math.}, pages 25--38. Birkh\"auser/Springer, New York, 2011.
\newblock \url{https://arxiv.org/abs/0710.2645}.

\bibitem{MR2820385}
Johannes Huebschmann.
\newblock The sh-{L}ie algebra perturbation lemma.
\newblock {\em Forum Math.}, 23(4):669--691, 2011.
\newblock \url{https://arxiv.org/abs/0710.2070}.

\bibitem{MR3584886}
Johannes Huebschmann.
\newblock Multi derivation {M}aurer--{C}artan algebras and sh {L}ie--{R}inehart
  algebras.
\newblock {\em J. Algebra}, 472:437--479, 2017.
\newblock \url{https://arxiv.org/abs/1303.4665}.

\bibitem{MR3881491}
Johannes Huebschmann.
\newblock The formal {K}uranishi parameterization via the universal homological
  perturbation theory solution of the deformation equation.
\newblock {\em Georgian Math. J.}, 25(4):529--544, 2018.
\newblock \url{https://arxiv.org/abs/1806.03225}.

\bibitem{MR1109665}
Johannes Huebschmann and Tornike Kadeishvili.
\newblock Small models for chain algebras.
\newblock {\em Math. Z.}, 207(2):245--280, 1991.

\bibitem{MR1932522}
Johannes Huebschmann and James~D. Stasheff.
\newblock Formal solution of the master equation via {HPT} and deformation
  theory.
\newblock {\em Forum Math.}, 14(6):847--868, 2002.
\newblock \url{https://arxiv.org/abs/math/9906036}.

\bibitem{MR0224678}
Thomas~W. Hungerford.
\newblock The free product of algebras.
\newblock {\em Illinois J. Math.}, 12:312--324, 1968.

\bibitem{MR0365571}
Dale Husemoller, John~C. Moore, and James Stasheff.
\newblock Differential homological algebra and homogeneous spaces.
\newblock {\em J. Pure Appl. Algebra}, 5:113--185, 1974.

\bibitem{MR0112157}
Kunihiko Kodaira, Louis Nirenberg, and Donald~C. Spencer.
\newblock On the existence of deformations of complex analytic structures.
\newblock {\em Ann. of Math. (2)}, 68:450--459, 1958.

\bibitem{MR0112154}
Kunihiko Kodaira and Donald~C. Spencer.
\newblock On deformations of complex analytic structures. {I}, {II}.
\newblock {\em Ann. of Math. (2)}, 67:328--466, 1958.

\bibitem{MR0195995}
Albert Nijenhuis and Roger~W. Richardson, Jr.
\newblock Cohomology and deformations in graded {L}ie algebras.
\newblock {\em Bull. Amer. Math. Soc.}, 72:1--29, 1966.

\bibitem{schlstas}
Michael Schlessinger and James~D. Stasheff.
\newblock Deformation theory and rational homotopy type.
\newblock {\em {\rm Preprint, new version}}, 2012.
\newblock \url{https://arxiv.org/abs/1211.1647}.

\bibitem{MR517083}
James~D. Stasheff.
\newblock Rational homotopy-obstruction and perturbation theory.
\newblock In {\em Algebraic topology ({P}roc. {C}onf., {U}niv. {B}ritish
  {C}olumbia, {V}ancouver, {B}.{C}., 1977)}, volume 673 of {\em Lecture Notes
  in Math.}, pages 7--31. Springer, Berlin, 1978.

\bibitem{MR1425752}
Willem~T. van Est.
\newblock Alg\`ebres de {M}aurer-{C}artan et holonomie.
\newblock {\em Ann. Fac. Sci. Toulouse Math.}, S\'erie 5(suppl.):93--134, 1989.

\end{thebibliography}

\end{document}